\documentclass[12pt]{amsart}

\usepackage{graphics}
\usepackage{amsmath}
\usepackage{amsfonts}
\usepackage{amssymb}
\usepackage{amscd}

\font\Bbb=msbm10 scaled \magstep 2

\def\C{\hbox{\Bbb C}}

\def\Z{\hbox{\Bbb Z}}
\def\N{\hbox{\Bbb N}}

\font\midBbb=msbm10

\def\midZ{\hbox{\midBbb Z}}

\newtheorem{theorem}{\bf Theorem}

\newtheorem{example}[theorem]{\bf Example}
\newtheorem{proposition}[theorem]{\bf Proposition}

\newtheorem{algorithm}[theorem]{\bf Algorithm}
\newtheorem{conjecture}[theorem]{\bf Conjecture}

\title{Linear differential operators for generic algebraic curves}

\author{V.A.\,Krasikov}
\address{Institute of Mathematics, \newline
\indent Siberian Federal University,  \newline
\indent 660041, Krasnoyarsk, Russia.}
\email{vkrasikov@sfu-kras.ru}
\thanks{Both authors were supported by the grant MK-3684.2009.1 of the President of Russian Federation.}

\author{T.M.\,Sadykov}
\address{Institute of Mathematics, \newline
\indent Siberian Federal University,  \newline
\indent 660041, Krasnoyarsk, Russia.}
\email{sadykov@lan.krasu.ru}
\thanks{T.M.\,Sadykov was supported
by the scholarship of the "Dynasty"\, foundation
and by Russian Foundation for Basic Research, grant 09-01-12132.}

\begin{document}

\large

\begin{abstract}
We give a computationally efficient method for constructing the linear
differential operator with polynomial coefficients whose space of holomorphic solutions is
spanned by all the branches of a function defined by a generic algebraic curve.
The proposed method does not require solving the algebraic equation and can be applied
in the case when its Galois group is not solvable.
\end{abstract}

\maketitle

\section{Introduction}

To find relations satisfied by a given special function is a difficult and
important problem in the theory of special functions of mathematical physics.
The relations in question can involve derivatives, integrals, finite differences etc.
Knowing a global relation for a special function that is defined locally (e.g. by means
of a series converging in a neighbourhood of a point) allows one to deduce global properties
of that function. From this point of view, linear differential equations with
polynomial coefficients are of particular interest.
One of the reasons for this is the difficult problem of computing the analytic continuation
along a given path of a locally defined special function.
By identifying such a function with a solution to a system of linear differential
equations with polynomial coefficients which does not have any "extra solutions"\,
one can use standard techniques for investigating the analytic continuation of the function
under study (see, for instance, \cite{PST}).
Here by "extra solutions"\, we mean the solutions which are not branches of the function
under study, that is, which cannot be obtained from it by means of analytic continuation.
Observe that every germ of a (multivalued) analytic function satisfies a relation with entire
(in particular, polynomial) coefficients provided that this relation is valid for one of its germs
in a neighbourhood of some fixed nonsingular point.

The culmination of this approach is the Wilf-Zeilberger algorithmic proof theory
(see \cite{WilfZeilberger1992} and \cite{Zeilberger1990}) based on holonomic systems of equations.
In the present paper we thoroughly investigate the special case when the function under study
is algebraic and the holonomic system consists of a single ordinary linear homogeneous
differential equation with polynomial coefficients.
Despite all simplicity, this setup leads to formidable computational challenges.

The 21st problem in the Hilbert list was solved in 1989 by A.A.\,Bolibrukh who proved that it is
in general not possible to construct a linear fuchsian system of differential equations with
a prescribed monodromy group (see \cite{AnosovBolibruch}).
However, the problem of effective computation of a system of differential equations (and, in
particular, of a single differential equation) with a prescribed branching of solutions
(whenever this is possible) remains open and is in the focus of intensive research,
see \cite{Bostan}, \cite{CarraFerro}, and \cite{Cormier}.
The computer algebra system {\it Magma} has a built-in command {\it DifferentialOperator}
for finding such operators (see \cite{Magma}).
In the paper \cite{LarussonSadykov} a method for computing annihilating operators for a class of algebraic
functions was developed.
However, due to computational difficulties and software limitations, it could only be used for computing annihilating operators
for algebraic functions defined by equations with solvable Galois groups.
In the present paper we describe an algorithm which allows one to compute annihilating operators
for an essentially larger class of algebraic functions and does not require the solvability of
the Galois group
(see Examples \ref{ex:determinationOfSimpleFunctions} (4), \ref{ex:y^5+ay^4+x}, \ref{ex:generic quintic} and \ref{ex:y^6+ay^2+by+x}).

It is well-known that an ordinary linear differential equation with a prescribed solution space
can be found by means of the wronskian of a basis of this space.
However, from the computational point of view, the wronskian-based representation of the differential
equation for an analytic function (which is, in general, defined only locally) is merely a nonconstructive
existence theorem.
There are three main reasons for this.
First, to form the wronskian, one needs to choose a basis
in the space of germs of the given function at a nonsingular point. This requires computing
the analytic continuation of the given function along any path which is, in general, a difficult problem.
Secondly, to evaluate a determinant containing high-order derivatives of a given special function is
a task of a great computational complexity.
Finally, extracting the polynomial coefficients of the desired differential operator out of the obtained
combination of algebraic functions requires the full use of modern methods of computer algebra.
For instance, to compute the differential operator for the roots of the generic monic cubic by means of
the wronskian is already a challenge (see example in \S 5 in \cite{LarussonSadykov}).
In the general case, the wronskian-based construction is not suitable for computation
since no effective means of simplifying expressions which contain high-order derivatives of
special functions are presently known.

The present paper provides an algorithm for constructing the optimal (that is, of the smallest
possible order) linear homogeneous differential equation with polynomial coefficients for a univariate
algebraic function $y=y(x)$ implicitly defined by the equation
\begin{equation}
y^m + a_1 y^{m_1} + \ldots + a_n y^{m_n} + x = 0.
\label{eq:mainAlgebraic}
\end{equation}
This algorithm allows one to overcome the difficulties listed above.
That is, it does not require to solve the problems of analytic continuation, evaluation of determinants
and simplification of expressions involving derivatives of algebraic functions.
The proposed method is a development of the ideas of the work \cite{LarussonSadykov}.
It allows one to reduce the problem of computing the annihilating operator for an algebraic
function to the problem of finding a basis in the syzygy module of an ideal in the ring
of multivariate polynomials.
The presented algorithm differs from other methods (both recent and classical,
see \cite{Bostan}, \cite{CarraFerro}, \cite{Cockle}, and \cite{Cormier}) in its primary field
of application (it deals with generic algebraic equations), in the underlying concept (holonomic
systems of partial differential equations and noncommutative elimination) and the complexity
of differential operators that it can efficiently produce.
The capabilities and limitations of the proposed algorithm are summarized in Table~1.

The authors are thankful to Professor D.\,Zeilberger for helpful explanations giving
insight into holonomic systems approach and to Professor M.\,Singer for comments on
Galois theory.


\section{Annihilating operators for solutions of holonomic systems of differential equations}

In what follows we will denote by $\mathcal{D}_n$ the Weyl algebra of differential operators
with polynomial coefficients in~$n$ variables $x=(x_1,\ldots,x_n)\in\C^n.$
This algebra is generated by the operators $x_1,\ldots ,x_n,$ $\frac{\partial}{\partial x_1},\ldots ,\frac{\partial}{\partial x_n}$
satisfying the relations
$\frac{\partial}{\partial x_i}\circ x_j - x_j \circ \frac{\partial}{\partial x_i} = \delta_{ij}.$
Here "$\circ$" denotes the composition of differential operators.
The Weyl algebra is simple (see Chapter~1 in \cite{Bjork1}).
When speaking about ideals in the Weyl algebra, we will always mean its left ideals.

The following basic statement is well-known but not easy to find in the literature in the following explicit form.
It can be deduced from Theorem~2 in \cite{Tsarev}.
It also follows from Theorem~1.4.12, Proposition~1.4.9, and Lemma~2.2.3 in \cite{SST}.
\begin{proposition}
For any holonomic left ideal $I\subset \mathcal{D}_n$ and any $i \in \{ 1,\ldots, n \}$ there exists
a nonzero operator $P_i \in I$,
all of whose derivatives are with respect to the variable $x_i$, that is, an operator of the form
$$ P_i  \, = \, \sum_{j=1}^{N_i} a_{ij} (x_1, \dots, x_n)  \,  \frac{\partial^j}{\partial x_{i}^{j}}.$$
\label{prop:annOpInHolonomicIdeal}
\end{proposition}

The following statement is a consequence of the results in \cite{gkz89} and \cite{Mayr}.
It shows that algebraic functions defined by generic algebraic curves are annihilated
by holonomic ideals in $\mathcal{D}_n.$

\begin{theorem}
Any germ of the algebraic function $y(x_0,x_1,$ $\ldots,$$x_n)$ implicitly defined by the relation
\begin{equation}
x_n y^n + x_{n-1} y^{n-1} + \ldots + x_1 y + x_0 = 0,
\label{eq:denseGenericAlgebraic}
\end{equation}
satisfies the holonomic system of differential equations
\begin{equation}
\begin{array}{l}
\frac{\partial^2\,y}{\partial x_i\partial x_j}=\frac{\partial^2\,y}{\partial x_k\partial x_l} \text{, whenever } i+j=k+l, \\
\sum\limits_{i=0}^{n}i\,x_i\,\frac{\partial y}{\partial x_i}=-y \quad \text{and} \quad \sum\limits_{i=0}^{n} x_i\,\frac{\partial y}{\partial x_i}=0.
\end{array}
\label{eq:MayrSystem}
\end{equation}
Conversely, any holomorphic solution of (\ref{eq:MayrSystem}) defined locally in a neighbourhood of a
nonsingular point is a linear combination of germs of the function $y(x_0,x_1,$ $\ldots,$ $x_n)$ at this point.
\label{thm:Mayr}
\end{theorem}
The system of differential equations (\ref{eq:MayrSystem}) is a special instance of
the Gelfand-Kapranov-Zelevinsky hypergeometric system introduced in \cite{gkz89}.
Its "dehomogenized" version for an algebraic curve with affine parameters
was investigated by Mellin in \cite{Mellin}.

Recall that the {\it Nilsson class} comprises (multi-valued) analytic functions of several
complex variables which have finite determination and moderate growth in arbitrary neighborhood
of any of their singularities (see \S 4.1.12 in \cite{Bjork2}).
Here by the determination of a multi-valued analytic function we mean the number of its linearly
independent germs in a neighborhood of a generic point in its domain of definition.
The determination of an analytic function of one complex variable which lies in the Nilsson
class and has finitely many singularities in $\C$ coincides with the smallest possible order
of an ordinary linear homogeneous differential equation with polynomial coefficients satisfied by this function.

Theorem \ref{thm:Mayr} together with Proposition \ref{prop:annOpInHolonomicIdeal} imply
the existence of a linear differential operator with polynomial coefficients whose space
of local holomorphic solutions at a nonsingular point is spanned by the roots of the
generic algebraic equation (\ref{eq:denseGenericAlgebraic}) and all of whose derivatives
are with respect to $x_0.$ This operator is defined uniquely up to a sign.
We will say that this operator is {\it optimal} for the given generic algebraic curve.

\begin{example}
\rm
Consider the algebraic function $y(x_0,x_1,x_2)$ defined as the solution to the quadratic equation
$x_2 y^2 + x_1 y + x_0 = 0.$
By Theorem \ref{thm:Mayr}, any of its branches lies in the kernel of any operator in the ideal $J$
with the generators
$$
A = \frac{\partial^2}{\partial x_0 \partial x_2} - \frac{\partial^2}{\partial x_{1}^2}, \hfill
B = x_1 \frac{\partial}{\partial x_1} + 2 x_2 \frac{\partial}{\partial x_2} + 1, \hfill
C = x_0 \frac{\partial}{\partial x_0} + x_1 \frac{\partial}{\partial x_1} + x_2 \frac{\partial}{\partial x_2}.
$$
Since the determination of the function $y(x_0,x_1,x_2)$ equals $2,$ Proposition \ref{prop:annOpInHolonomicIdeal}
yields the existence of a second order differential operator $P\in J,$ all of whose derivatives are with
respect to $x_0.$
Using the notation $\theta_i = x_i \frac{\partial}{\partial x_i},$ we can write the expansion of this operator
with respect to the basis of $J$ in the form
$$
P = x_0 x_{1}^2 x_2 A - ( (x_{1}^2 - 2 x_0 x_2) \theta_{0} + x_0 x_2 \theta_1 ) B
+ ( (x_{1}^2 - 4 x_0 x_2) \theta_0 + 2 x_0 x_2 \theta_1 ) C
$$
$$
=x_{0}^2 \left( (x_{1}^2 - 4 x_0 x_2) \frac{\partial^2}{\partial x_{0}^2}  - 2 x_2 \frac{\partial}{\partial x_0} \right).
$$
Of course, this optimal differential operator is only a monomial multiple of the wronskian of the roots
of the initial algebraic equation.
\label{ex:gkzIntoOrdinaryOperator}
\end{example}
In the next section we describe the algorithm for computing the optimal annihilating operator
for an arbitrary algebraic function satisfying an equation of the form (\ref{eq:mainAlgebraic}).
This will, in particular, perform the noncommutative elimination of all the derivatives except for $\frac{\partial}{\partial x_i}$
in the holonomic ideal (\ref{eq:MayrSystem}) by means of methods of commutative algebra only.


\section{Computing the annihilating operator for a given algebraic function}

We begin by computing the determinations of some elementary functions and
the corresponding differential equations.

\begin{example}
\rm
(1) Any rational function $f=p(x)/q(x),$ where $p(x),$ $q(x)\in\C[x],$ has  determination~$1$
and satisfies the first-order differential equation $pq f' = (p'q - p q')f.$

\noindent
(2) The function $f=x^a$ also has determination~$1$ for any $a\in\C$ since its analytic continuation
$e^{2 \pi i a}f$ around the only finite singularity $x=0$ is proportional to $f.$
It satisfies the differential equation $xf'=af.$

\noindent
(3) The function $f=\ln x$ has determination~$2,$ since its analytic continuation along any path
can be written in the form $\ln x + 2\pi k i,$ $k\in\Z.$ Thus any germ of $f$ at a nonsingular point
lies in the two-dimensional linear space with the basis $\{1,\ln x\}.$ The second-order differential
equation with polynomial coefficients satisfied by $f$ has the form $x f''+f'=0.$

\noindent
(4) The algebraic function $y=y(x)$ implicitly defined by the relation $y^5 + a y + x =0$ has determination $4$
(see Theorem \ref{thm:determination} below) and satisfies the differential equation
$(256 a^5 + 3125 x^4) y^{(4)} + 31250 x^3 y^{(3)} + 73125 x^2 y^{(2)}  + 31875 x y^{'} -1155 y = 0.$

\noindent
(5) Finally, the function $1/\ln x$ has infinite determination since its germs
$\left\{ 1/(\ln x + 2 \pi k i) \right\}_{k\in\midZ}$ are linearly independent.
This implies, in particular, that this function does not satisfy any linear homogeneous
differential equation with polynomial coefficients.

\label{ex:determinationOfSimpleFunctions}
\end{example}

In the present section we describe an algorithm for computing the optimal annihilating operator
for the roots of a generic algebraic equation with symbolic coefficients, that is,
an equation of the form (\ref{eq:mainAlgebraic}).
The roots of the equation $a_0 y^m + a_1 y^{m_1} + a_2 y^{m_2} + \ldots + a_n y^{m_n} + a_{n+1} =0$
(regarded as functions of $a = (a_0, \dots, a_{n+1})$) satisfy the holonomic
$A$-hypergeometric system with the vector of parameters $(0, -1)$
(see \cite{Sturmfels}), where
$$
A := \, \left(
\begin{array}{ccccc}
1 & 1& \dots &1&1 \\
m & m_1 & \dots & m_n & 0
\end{array}
\right).
$$
Namely, it is the left ideal in the Weyl algebra
$\C[a_0,\dots,a_{n+1},\partial_0,\dots,\partial_{n+1}]$ generated by
$$
\hbox{the toric operators} \quad
\partial^u - \partial^v \quad \hbox{for}\quad
 u,v\in \N^{n+2}   \quad \hbox{with}\quad A\cdot u=A\cdot v,
$$
$$
\label{eq:firstorder} \,\, \hbox{and the Euler operators} \quad
\sum_{j=0}^{n+1} a_j \partial_j \quad \hbox{and} \quad m a_0 +
\sum_{j=1}^{n} m_j a_j \partial_j +1.
$$
Thus by Proposition \ref{prop:annOpInHolonomicIdeal}, there always exists a linear differential operator
with polynomial coefficients in $a_0,\ldots,a_{n+1},$ and all of whose derivatives are with respect to $a_{n+1}.$
Setting $a_0 = 1$ and $a_{n+1} = x$ we obtain the annihilating operator for the solutions of (\ref{eq:mainAlgebraic}).
Using noncommutative elimination theory, one can compute this operator in a way
similar to that in Example \ref{ex:gkzIntoOrdinaryOperator}.
In the special case of a trinomial equation (that is, for $n=1$) the desired operator
is a right factor of the Mellin differential operator found in \cite{Mellin}.

The following theorem gives the order of the annihilating operator.
\begin{theorem}{\rm (E.\,Cattani, C.\,D'Andrea, A.\,Dickenstein \cite{CDD})}
The number of linearly independent (over the field of complex numbers) germs
of the solutions to the equation (\ref{eq:mainAlgebraic}) at a generic point $x\in\C$
and for generic values of the parameters $(a_1,\ldots, a_n)\in \C^n$ is given by
$$
R(m,m_1, \ldots, m_n) =
\left\{
\begin{array}{lr}
m - 1 + \left[ \frac{m_1}{m-1} \right],        & \text{\ \ if\ \ }  \text{\rm GCD}(m, m_1,\ldots, m_n) = 1,           \\
\frac{m}{\text{\rm GCD}(m, m_1,\ldots, m_n)},  & \text{\ \ if\ \ }  \text{\rm GCD}(m, m_1,\ldots, m_n) > 1.
\end{array}
\right.
$$
Here $[\,]$ denotes the integer part of a real number.
\label{thm:determination}
\end{theorem}

The following theorem is the foundation of our algorithm for computing optimal annihilating operators.

\begin{theorem}
Let $s_{i} = s_{i}(x, a_1, \ldots, a_{n}),$ $i=1, \ldots, m$ be the roots of the algebraic equation (\ref{eq:mainAlgebraic})
and denote $P(t)= \prod_{i=1}^m (t - s_i).$
For every $k=1, \ldots, m$ we define the ideal $I_k$ in the polynomial ring
with $m+n+1$ variables $\C[s_1,$ $\ldots,$ $s_m,$ $a_1,$ $\ldots,$ $a_n, x]$ to be
$$
      \left(
      (-1)^\ell
      (\ell - 1)!
      \left( \prod_{i \neq k} (s_k - s_i)^{2 m - 1} \right) \,
      \underset{t = s_k}{\rm res} \, \frac{1}{{P(t)}^\ell}, \,\,
      \ell = 1,\ldots, R(m,m_1,\ldots,m_n)
      \right).
$$
The vector of polynomial coefficients of the optimal annihilating operator for the algebraic
function defined by (\ref{eq:mainAlgebraic}) lies in the following syzygy module of the quotient of the ideal $I_k$
with respect to the Vieta relations:
\begin{equation}
\begin{array}{ll}
{\rm Syz}
\Big(
I_k \Big/ \big(
                 & \mathcal{S}_{m - m_j} (s_1, \ldots, s_m) - (-1)^{m - m_j} \, a_{j}, \, j=1, \ldots, n;   \\
                 & \mathcal{S}_{m - k} (s_1, \ldots, s_m), \, {\text\ for\ } k\not\in \{ 0, m_1,m_2, \ldots, m_n \}; \\
                 & \mathcal{S}_{m} (s_1, \ldots, s_m) - (-1)^m x
\big) \Big).
\end{array}
\label{eq:syzygy}
\end{equation}
\label{thm:syzygies}
Here $\mathcal{S}_j (s_1, \ldots, s_m)$ is the elementary symmetric polynomial of order $j$ in the variables $(s_1,\ldots,s_m).$
In the sequel we will denote the ideal generated by the Vieta relations by $\mathcal{V}.$
\end{theorem}
\begin{proof}
Let $x\in\C$ be a point outside of the zero locus of the discriminant of the left-hand side in (\ref{eq:mainAlgebraic}).
Let $(a_1,\ldots, a_n)\in \C^n$ be a generic vector of parameters and let
$y_{k} (x,a_1,\ldots,a_n)$ denote the $k$-th branch of the solution $y(x,a_1,\ldots,a_n)$ to the algebraic equation (\ref{eq:mainAlgebraic}).
Let us now denote by~$D$ the differential operator
$$
D=\frac{\partial^m}{\partial s_1 \ldots \partial s_m}.
$$
Using the well-known contour integral representation for a solution to a univariate algebraic
equation (see Section 5 in \cite{LarussonSadykov}) we conclude that the generators of the ideal $I_k$
are polynomial multiples of the derivatives of the solution to (\ref{eq:mainAlgebraic}):
$$
\begin{array}{l}
\frac{\partial^\ell y_{k}(x,a_1,\ldots,a_n)}{\partial x^\ell} \\
= (-1)^\ell (\ell - 1)! \, \underset{t = s_k}{\rm res} \, \frac{1}{{P(t)}^\ell}
= \frac{(-1)^\ell}{((\ell - 1)!)^{m-1}} D^{\ell - 1} \left( \underset{t = s_k}{\rm res} \, \frac{1}{P(t)} \right)                 \\
= \frac{(-1)^\ell}{((\ell - 1)!)^{m-1}} D^{\ell - 1} \left( \frac{1}{(s_k - s_1) \ldots [k] \ldots (s_k - s_n)} \right)           \\
= (-1)^\ell \frac{\partial^{\ell - 1}}{\partial s_{k}^{\ell - 1}} \frac{1}{{((s_k - s_1) \ldots [k] \ldots (s_k - s_m))}^\ell}    \\
= -\sum\limits_{i_1 + \ldots [k] \ldots + i_n = \ell - 1}
\frac{(\ell + i_1 - 1)!\ldots [k] \ldots (\ell + i_n - 1)!}
{((\ell - 1)!)^{m-2} \, i_1! \ldots [k] \ldots i_m! \,
(s_k - s_1)^{\ell + i_1} \ldots [k] \ldots (s_k - s_n)^{\ell + i_n}}. \\
\end{array}
$$
This shows that the generators of the ideal $I_k$ are indeed elements of the ring
$\C[s_1,$ $\ldots,$ $s_m,$ $a_1,$ $\ldots,$ $a_n, x].$
By Theorem~\ref{thm:determination} the determination of the solution to (\ref{eq:mainAlgebraic}) equals $R(m,m_1,$ $\ldots,$ $m_n).$
Thus by Theorem~\ref{thm:Mayr} and Proposition~\ref{prop:annOpInHolonomicIdeal} there exists a
linear differential operator with polynomial coefficients (in $x,a_1,\ldots,a_n$) all of whose derivatives are with respect to $x$
and whose space of holomorphic solutions at a generic point is spanned by the branches of $y(x,a_1,\ldots,a_n)$.
For the sake of computational efficiency we factor out the Vieta relations.
This increases the number of variables involved in the generators of the ideal but decreases their degrees.
The desired differential operator is a relation between the derivatives of $y(x,a_1,\ldots,a_n)$ with polynomial
(and thus single-valued) coefficients. By the conservation principle for analytic functions the
same relation must be satisfied by any of the germs of $y(x,a_1,\ldots,a_n)$ at a nonsingular point.
Thus the coefficients of this relation lie in the syzygy module (\ref{eq:syzygy}).
\end{proof}

Observe that the elements of the syzygy module (\ref{eq:syzygy}) are polynomial vectors whose entries in general
depend on all of the variables $s_1,\ldots,s_m,$ $a_1, \ldots, a_n, x$.
The proof of Theorem \ref{thm:syzygies} implies that there exists an element of (\ref{eq:syzygy}) whose
entries only depend on $a_1, \ldots, a_n, x$.
It can be found by means of the following algorithm.

\begin{algorithm}
\rm
The actual computation of annihilating operators for algebraic functions was organized as follows:

1. Compute the basis of the ideal $I_1$ defined in Theorem \ref{thm:syzygies}.

2. Using the lexicographic order of the variables $s_1,\ldots, s_m$ compute the Gr\"obner basis of the ideal $\mathcal{V}$
defined by the Vieta relations (as defined in Theorem \ref{thm:syzygies}).

3. Perform polynomial reduction of the generators of the ideal $I_1$ by means of the Gr\"obner basis of the ideal $\mathcal{V}$.
That is, at this step, we use the Vieta relations as much as possible in order to simplify the generators of $\mathcal{V}$.

4. Factorize the obtained family of polynomials. The result has a very specific structure: it is a family
of polynomials in $\C[s_1,$ $\ldots,$ $s_m,$ $a_1,$ $\ldots,$ $a_n, x]$ whose elements are symmetric
with respect to $s_2,\ldots, s_m.$
Using the Gr\"obner basis of the ideal $\mathcal{V},$ reduce them to polynomials
in $\C[s_1,$ $a_1,$ $\ldots,$ $a_n, x].$
Let us denote this family of polynomials by $\mathcal{R}_1,\ldots, \mathcal{R}_m.$

5. Any $\C[a_1,\ldots,a_n, x]$-linear relation for the family of polynomials $\mathcal{R}_1,$ $\ldots,$ $\mathcal{R}_m$
transforms into a linear system of algebraic equations over the field of rational functions in the variables $a_1,\ldots,a_n, x.$
Proposition \ref{prop:annOpInHolonomicIdeal} and Theorem \ref{thm:determination} yield the existence of
an at least one-dimensional $\C$-vector space of solutions to this system of linear equations.
Finding a basis in this space and clearing the denominators, we obtain the desired polynomial coefficients
of the optimal annihilating operator for the initial algebraic function.~\hfill $\square$
\label{alg:main}
\end{algorithm}

\begin{example}
\rm
The linear space spanned by the roots of the algebraic equation
\begin{equation}
y^5 + 2y^4 - 3y^3 + y^2 + 5y + x = 0
\label{eq:y^5+2y^4-3y^3+y^2+5y+x=0}
\end{equation}
(in a neighbourhood of a point where the discriminant of this equation does not vanish)
coincides with the linear space of holomorphic solutions to the differential equation
{\small
$$
\begin{array}{l}
(-43728190560 + 795819153\, x - 53446888\, x^2 + 56028\, x^3)   \\
 (-1585575 + 71982\, x + 281583\, x^2 + 81342\, x^3 + 3125\, x^4) \,\, y^{(5)} + \\
15 (-650327879439783 - 5747872136026563\, x - 2400588229818366\, x^2 -   \\
 91559102743545\, x^3 - 304019551433\, x^4 - 131338505212\, x^5 + 128397500\, x^6) \,\, y^{(4)} + \\
60 (-1821690090417321 - 1560609625036728\, x - 98711280942848\, x^2 +  \\
 721492325057\, x^3 - 103787727624\, x^4 + 91045500\, x^5) \,\, y^{(3)} + \\
180(-282046871305467 - 38794189010031\, x + 478890241959\, x^2  -  \\
 28458003540\, x^3 + 21944300\, x^4) \,\, y^{(2)} + \\
 720(-1756652589603 + 23053844253\, x - 812236372\, x^2 + 522928\, x^3) \,\, y' =  0.
\end{array}
$$
} 
Observe that the Galois group of the algebraic equation (\ref{eq:y^5+2y^4-3y^3+y^2+5y+x=0})
is not solvable and hence its roots can only be expressed
in terms of special functions (e.g., functions of hypergeometric type, see \cite{Sturmfels}).
Despite this fact, Algorithm \ref{alg:main} allows one to compute the annihilating operator
for the roots of the equation (\ref{eq:y^5+2y^4-3y^3+y^2+5y+x=0}) using only the methods of
commutative algebra.
\label{ex:y^5+2y^4-3y^3+y^2+5y+x}
\end{example}

The following example provides a fundamental system of solutions to a fifth-order
linear differential equation with polynomial coefficients.

\begin{example}
\rm
For any $a\in\C^{*}$ a basis in the space of holomorphic solutions to the differential equation
$$
((256/5)\, a^5 x^3 + 625\, x^4) y^{(5)} + (384\, a^5 x^2 + 6875\, x^3) y^{(4)}
\phantom{--------}
$$
$$
\phantom{---}
+ (624\, a^5 x + 19500\,  x^2) y^{(3)}
+ (168\, a^5 + 14100\, x) y^{(2)} + 1344 y^{'} =0
$$
in a neighbourhood of a generic point $x\in\C$ is given by the roots of the algebraic equation
$ y^5+ay^4+x=0. $
\label{ex:y^5+ay^4+x}
\end{example}


\section{Software, hardware and examples}

Most of the examples in this paper were computed by means of a {\it Mathematica~6.0} package
developed by the authors and run on an Intel Core(TM) Duo CPU clocked at 2.00GHz.
The bottleneck of the algorithm is computing the syzygy module of an ideal in a ring of polynomials in several variables.
In some cases, we have used {\it Singular} for this.

\begin{example}
Generic quintic.
\rm
One of the goals of the research presented in this paper was to provide a computationally efficient
extension of the results of Section~5 in \cite{LarussonSadykov} beyond the class of algebraic
equations with elementary solutions. In this example, we demonstrate the efficiency of the described
approach by means of the generic monic quintic
\begin{equation}
y^5 + a_4 y^4 + a_3 y^3 + a_2 y^2 + a_1 y + x = 0.
\label{eq:generic quintic}
\end{equation}
Computing the annihilating operator for the solutions of this equation
has turned out to be a task of considerable computational complexity.
The full output of the algorithm is a vector of five polynomials with $ 4306 $ monomials
in total and is too large to display. The degrees of these polynomials
with respect to the variables $a_1, \ldots, a_4, x$ are $15, 20, 21, 22, 23.$
The leading coefficient of the annihilating operator has degree $7$ with respect to $x$
and splits into the product of two factors. One of them is the discriminant of (\ref{eq:generic quintic})
while the other is a polynomial of total degree $15$ with $ 264 $ terms.
The largest of the numeric coefficients in the annihilating operator for the generic monic quintic
equals $ 2739594525000.$
\label{ex:generic quintic}
\end{example}

\begin{example}
A monic tetranomial with generic coefficients.
\rm
By Theorem \ref{thm:determination} the determination of a solution to the algebraic equation
\begin{equation}
y^6 + ay^2 + by + x = 0
\label{eq:y^6+ay^2+by+x}
\end{equation}
equals five. The roots of (\ref{eq:y^6+ay^2+by+x}) at a generic point $x\in\C$ span the
space of holomorphic solutions of the following fifth order linear differential operator
with polynomial coefficients:
{\small
$$
\begin{array}{l}
(-255664128\, a^{10} + 395740000\, a^5 b^4 + 1599609375\, b^8 + 148780800\, a^6 b^2 x + 2859609375\, a b^6 x - \\
499654656\, a^7 x^2 - 1573425000\, a^2 b^4 x^2 + 1051704000\, a^3 b^2 x^3 + 16796160\, a^4 x^4) \\
(256\, a^5 b^2 + 3125\, b^6 - 1024\, a^6 x - 22500\, a b^4 x + 43200\, a^2 b^2 x^2 - 13824\, a^3 x^3 - 46656\, x^5) \frac{d^5}{dx^5} +
\end{array}
$$
$$
\begin{array}{l}
(916300234752\, a^{16} + 18677130035200\, a^{11} b^4 - 38094525000000\, a^6 b^8 - 134905517578125\, a b^{12} - \\
77437887971328\, a^{12} b^2 x + 107910691200000\, a^7 b^6 x + 332702753906250\, a^2 b^{10} x + \\
37877629059072\, a^{13} x^2 - 406052352000\, a^8 b^4 x^2 + 552267618750000\, a^3 b^8 x^2 - \\
128020162314240\, a^9 b^2 x^3 - 727448202000000\, a^4 b^6 x^3 + 267464667561984\, a^{10} x^4 + \\
43693344000000\, a^5 b^4 x^4 - 1306049062500000\, b^8 x^4 - 221398918963200\, a^6 b^2 x^5 - \\
2201395927500000\, a b^6 x^5 + 359825022517248\, a^7 x^6 + 1137850610400000\, a^2 b^4 x^6 - \\
711490376448000\, a^3 b^2 x^7 - 10579162152960\, a^4 x^8) \frac{d^4}{dx^4} +
\end{array}
$$
$$
\begin{array}{l}
30 (-3264411795456\, a^{12} b^2 + 5653930240000\, a^7 b^6 + 24016248046875\, a^2 b^{10} + \\
3132022849536\, a^{13} x - 5896377011200\, a^8 b^4 x - 145456875000\, a^3 b^8 x - \\
2678330105856\, a^9 b^2 x^2 - 31473123300000\, a^4 b^6 x^2 + 38750783864832\, a^{10} x^3 - \\
41250841920000\, a^5 b^4 x^3 - 221475515625000\, b^8 x^3 - 21927996518400\, a^6 b^2 x^4 - \\
344319609375000\, a b^6 x^4 + 52068310990848\, a^7 x^5 + 163366013160000\, a^2 b^4 x^5 - \\
95796142732800\, a^3 b^2 x^6 - 1311148560384\, a^4 x^7) \frac{d^3}{dx^3}
\end{array}
$$
$$
\begin{array}{l}
-120 (-441539395584\, a^{13} + 907546908800\, a^8 b^4 + 3520234921875\, a^3 b^8 - \\
122881784832\, a^9 b^2 x + 2644039462500\, a^4 b^6 x - 14126136238080\, a^{10} x^2 + \\
20651647680000\, a^5 b^4 x^2 + 88754326171875\, b^8 x^2 + 4397772787200\, a^6 b^2 x^3 + \\
122488079765625\, a b^6 x^3 - 16700798121984\, a^7 x^4 - 51772358925000\, a^2 b^4 x^4 + \\
28447624699200\, a^3 b^2 x^5 + 352185242112\, a^4 x^6) \frac{d^2}{dx^2}
\end{array}
$$
$$
\begin{array}{l}
-2520 (4704614400\, a^9 b^2 + 4079375000\, a^4 b^6 - 219268374528\, a^{10} x + \\
350501380000\, a^5 b^4 x + 1464693750000\, b^8 x + 5043513600\, a^6 b^2 x^2 + \\
1670568609375\, a b^6 x^2 - 196541448192\, a^7 x^3 - 597319515000\, a^2 b^4 x^3 + \\
304660958400\, a^3 b^2 x^4 + 3302125056\, a^4 x^5) \frac{d}{dx} +
\end{array}
$$
$$
\begin{array}{l}
5040 (-2372960256\, a^{10} + 4421840000\, a^5 b^4 + 18235546875\, b^8 - 1113523200\, a^6 b^2 x + \\
14120437500\, a b^6 x - 1226244096\, a^7 x^2 - 3581820000\, a^2 b^4 x^2 + \\
1587859200\, a^3 b^2 x^3 + 13436928\, a^4 x^4).
\end{array}
$$
}

\label{ex:y^6+ay^2+by+x}
\end{example}

Recall that the Newton polytope of a multivariate Laurent polynomial
$$
f(x_1,\ldots,x_n) = \sum_{\alpha = (\alpha_1,\ldots,\alpha_n) \in A \subset \midZ^n} c_\alpha \, x_{1}^{\alpha_1} \ldots x_{n}^{\alpha_n}
$$
supported in a finite set $A$ is defined to be the convex hull of $A.$
Intensive experiments suggest that there is an intrinsic relation between the
two extreme coefficients in the optimal annihilating differential operator for an
algebraic function. As we have seen in several examples before, the leading
coefficient in the annihilating operator is typically given by the product
of the discriminant of the defining algebraic equation and some other factor which
has no apparent relation to the initial algebraic equation.
The following conjecture suggests the structure of the Newton polytope of this polynomial
factor.
\begin{conjecture}
Let
$$
\sum_{k = \ell}^{d} p_{k}(a_1, \ldots, a_{n},x) \frac{d^k}{dx^k}
$$
be the optimal annihilating operator for the algebraic function defined by the relation
$P(x,y) := y^m + a_{1} y^{m_1} + \ldots + a_n y^{m_n} + x = 0.$
Denote by $\mathfrak{D}(a_1, \ldots, a_{n},x)$ the discriminant of
$P(x,y)$ computed with respect to $y.$
Then the polynomials
$$
p_{d}(a_1, \ldots, a_{n},x)/\mathfrak{D}(a_1, \ldots, a_{n},x) \text{\ and\ } p_{\ell}(a_1, \ldots, a_{n},x)
$$
consist of monomials with the same exponent vectors.
In particular, they contain equally many monomials and have equal Newton polytopes.
\label{conj:Newton polytope of parasite factor}
\end{conjecture}

The following table summarizes the results of our computer experiments and illustrates Conjecture \ref{conj:Newton polytope of parasite factor}.
It gives the linear ordinary differential operator whose solution space
is spanned by the branches of an implicitly defined algebraic function $y=y(x),$
the order of this operator and the multidegree of its leading coefficient with respect to $x$
and the parameters of equation listed in lexicographic order.
\vfill

\newpage
{\small
\centerline{\bf Table~1: Computation times for annihilating operators and their properties}
\vskip0.2cm
\noindent
\begin{tabular}{|c|c|c|c|c|}
\hline
Algebraic curve   &  The annihilating operator for $y=y(x)$                         & Order & \begin{minipage}{1.3cm} Leading coeff. \end{minipage} & \begin{minipage}{1.8cm} Comput. time (sec.) \end{minipage} \\
\hline
$y^4+a y^3 + x = 0$         & $(27 a^4 x^2 - 256 x^3) \frac{d^4}{dx^4} + 4 x(27 a^4 - 416 x) \frac{d^3}{dx^3}$   &  $4$  &   (2,4) & 0.374        \\
                            & $ + 60(a^4 - 36 x) \frac{d^2}{dx^2} -360 \frac{d}{dx}$      &       &                &              \\
\hline
                            & $x(14 b^3 - 4 a^2 b^2 + 8 b x - 3 a^2 x)(16 b^4 - 4 a^2 b^3 -$         &       &            &              \\
$y^4+a y^3 + b y^2+x=0$     & $ 128 b^2 x + 144 a^2 b x - 27 a^4 x + 256 x^2) \frac{d^4}{dx^4}$                    &  $4$ &  (4,6,7) & 1.03  \\
                            & $+ \ldots+120(74 b^3 - 21 a^2 b^2 + 24 b x - 9 a^2 x)\frac{d}{dx}$ &       &                &              \\
\hline
                            & $(45 c^2 + 14 b^3 - 47 a b c - 4 a^2 b^2 + 12 a^3 c + 8 b x$&       &                &              \\
$y^4+a y^3 + b y^2 + $      & $- 3 a^2 x)(27 c^4 +\ldots-256 c^3) \frac{d^4}{dx^4} +$     &  $4$  &  (4,7,7,6)     &  3.011       \\
$c y+x=0$                   & $\ldots+120(243 c^2 + 74 b^3 - 249 a b c - 21 a^2 b^2  $    &       &                &              \\
                            & $+ 63 a^3 c + 24 b x - 9 a^2 x) \frac{d}{dx} $              &       &                &              \\
\hline
$y^5 + a y + x = 0$         & $(256 a^5 + 3125 x^4) \frac{d^4}{dx^4} + 31250 x^3 \frac{d^3}{dx^3} +$&  $4$ & (4,5) &  2.012       \\
                            & $73125 x^2 \frac{d^2}{dx^2} + 31875 x \frac{d}{dx} - 1155$   &      &                &              \\
\hline
                            & $ (51200 b^6 - 15930 a^4 b^3 - 2187 a^8 + 68000 a b^4 x   $ &       &                &              \\
                            & $ -1350 a^5 b x - 26250 a^2 b^2 x^2) (256 b^5 - $           &       &                &              \\
$y^5 + a y^2 + b y + x=0$   & $ 27 a^4 b^2 - 1600 a b^3 x + 108 a^5 x +2250 a^2 b x^2$    &  $4$  &  (6,13,11)     &  5.008       \\
                            & $  + 3125 x^4) \frac{d^4}{dx^4} +\ldots$                    &       &                &              \\
                            & $+ 120(492800 b^6 - 139995 a^4 b^3 - 16038 a^8 +$           &       &                &              \\
                            & $222000 a b^4 x + 8100 a^5 b x - 39375 a^2 b^2 x^2)$        &       &                &              \\
\hline
$y^5 + a y^3 + b y^2 + c y$ & $(102400 c^6+\ldots-2500 a^3 b x^3)(256 c^5 - 27 b^4 c^2  $ &       &  (7,15,        &              \\
$+x=0$                      & $ + 144 a b^2 c^3 +\ldots+3125 x^4)\frac{d^4}{dx^4}$        &  $4$  &   13,11)       & 31.668       \\
                            & $ +\ldots -120(985600 c^6+\ldots-625 a^3 b x^3)$            &       &                &              \\
\hline
                            & $(265 a^5 x^3 + 3125 x^4) \frac{d^5}{dx^5} + 5 x^2 $        &       &                &              \\
$y^5 + a y^4 + x = 0$       & $(384 a^5 + 6875 x) \frac{d^4}{dx^4} + 780 x(4 a^5 + 125 x) \frac{d^3}{dx^3}   $   & $5$ & (4,5) & 10.436  \\
                            & $+ 60(14 a^5 + 1175 x) \frac{d^2}{dx^2}+ 6720 \frac{d}{dx}$ &       &                &              \\
\hline
                            & $(1680 a b^9 +\ldots+56 a^9 x^2)(108 b^5 x^2 -    $         &       &                &              \\
$y^5 + a y^4 + b y^3 + x=0$ & $27 a^2 b^4 x^2 + 2250 а b^2 x^3 - 1600 a^3 b x^3 + $       &  $5$  &    (6,15,14)   & 38.563       \\
                            & $256 a^5 x^3 + 3125 x^4) \frac{d^5}{dx^5}+\ldots+$          &       &                &              \\
                            & $1680(20160 a b^9 +\ldots+224 a^9 x^2)\frac{d}{dx}$         &       &                &              \\
\hline
                            & $(160380 c^8 +\ldots- 56 a^9 x^3)(108 c^5 x +$              &       &                &              \\
$y^5 + a y^4 + b y^3 +$     & $  16 b^3 c^3 x +\ldots+ 3125 x^4) \frac{d^5}{dx^5}   $     &  $5$  &  (7,16,        &  279.242     \\
$c y^2 + x = 0$             & $+\ldots+1680(2779920 c^8 - 1242720  b^3 c^6 +\ldots  $     &       &   14,13)       &              \\
                            & $- 224 a^9 x^3)\frac{d}{dx}$                                &       &                &              \\
\hline
$y^5 + a y^4 + b y^3 +$     & $(4928000 d^6 + \ldots + 56 a^9 x^3)(256 d^5 - 27 c^4 d^2$  &       & (7,17,15,      &              \\
$c y^2 + d y + x = 0$       & $ + 144 b c^2 d^3 + \ldots + 3125 x^4) \frac{d^5}{dx^5}$    &  $5$  & 13,11)         &  3038.47     \\
                            & $ + \ldots+1680(9011200 d^6+\ldots+224 a^9 x^3)\frac{d}{dx}$&       &                &              \\
\hline
                            & $(6398437500 c^{10} + \ldots + 67184640 a^2 b c^2 x^5)    $ &       &                &              \\
$y^6 + a y^3 + b y^2 +    $ & $(3125 c^6 + 256 b^5 c^2 + \ldots                         $ &  $5$  & (10,18,        &  799.427     \\
$c y + x = 0$               & $ - 46656 x^5)\frac{d^5}{dx^5} + \ldots + 5040            $ &       &  17,16)        &              \\
                            &   $(72942187500 c^{10} + \ldots + 13436928 a^2 b c^2 x^5) $ &       &                &              \\
\hline
\end{tabular}

}

\end{document}